\documentclass{amsart}
\usepackage{amssymb,latexsym, amsmath, amscd, array, graphicx}

\usepackage{graphicx}
\usepackage[all]{xy}
\usepackage{amscd}
\swapnumbers
\numberwithin{equation}{section}

\theoremstyle{plain}

\newtheorem{thm}{Theorem}[section]
\newtheorem{cor}[thm]{Corollary}

\newtheorem{prop}[thm]{Proposition}

\theoremstyle{definition}
\newtheorem{defin}[thm]{Definition}

\newtheorem{remark[thm]}{Remark}

\def\TC{{\mathop\mathrm{TC}\,}}

\def\cat{\protect\operatorname{cat}}

\def\cd{\protect\operatorname{cd}}

\def\Z{{\mathbb Z}}

\def\R{{\mathbb R}}

\def\1{\hbox{\rm\rlap {1}\hskip.03in{\rom I}}}
\def\Bbbone{{\rm1\mathchoice{\kern-0.25em}{\kern-0.25em}
{\kern-0.2em}{\kern-0.2em}I}}

\long\def\forget#1\forgotten{} %

\begin{document}

\title[On topological complexity of hyperbolic groups]
{On topological complexity of hyperbolic groups}
\author[A.~Dranishnikov]
{Alexander Dranishnikov}
\address{A. Dranishnikov, Department of Mathematics, University
of Florida, 358 Little Hall, Gainesville, FL 32611-8105, USA}
\email{dranish@math.ufl.edu}

\begin{abstract} We show that the topological complexity of a finitely generated torsion free hyperbolic group $\pi$ with $\cd\pi=n$ equals $2n$.
\end{abstract}

\maketitle

\section{Introduction}

The topological complexity of a space is a numerical invariant defined by M. Farber~\cite{F} in his study of motion planning in robotics.
 The {\em topological complexity}  $\TC(X)$ of a space $X$ is the minimum of $k$ 
such that $X\times X$ admits an open cover $U_0,\dots, U_k$ with the property that over each $U_i$ there is a continuous  motion planning algorithm $s_i$, i.e. a continuous map $s_i:U_i\to PX$ to the path space  $PX=X^{[0,1]}$ such that $s_i(x,y)(0)=x$ and $s_i(x,y)(1)=y$ for all $(x,y)\in U_i$. 

The topological complexity is a homotopy invariant in spirit of the Lusternik-Schnirelmann category $\cat(X)$. Like in the case of the LS-category, 
generally this invariant is hard to compute. Since $\TC(X)$ is homotopy invariant, one can define $\TC(\pi)$, the topological complexity of a discrete group $\pi$
as $\TC(\pi)=\TC(B\pi)$. This definition works for the LS-category as well. Both group invariants turns to be $\infty$ in the presence of torsions. So we consider 
torsion free groups only.
Eilenberg and Ganea~\cite{EG} proved that the LS-category of a discrete group equals its cohomological dimension,
$\cat(\pi)=\cd(\pi)$. It turns out that in the case of topological complexity the range of $\TC(\pi)$ is between $\cd(\pi)$ and $2\cd(\pi)$ (see~\cite{CP},\cite{Ru},\cite{DS}).

The minimal value of  $\TC$ is taken on abelian groups, $\TC(\pi)=\cd(\pi)$.
Combining results of ~\cite{DS} and \cite{Dr1} present the cases 
of the maximal value $\TC(\pi)=2\cd(\pi)$ for groups which are free "square",
$\pi=H\ast H$, of a geometrically finite group. In this paper we investigate further  the question when the topological complexity of a group is maximal.
We consider groups $\pi$ with finite classifying complex  $B\pi$.  In particular, in this paper
we prove the equality $\TC(\pi)=2\cd(\pi)$ for hyperbolic groups.
A giant step in computation of $\TC$ for hyperbolic groups was made by Farber and Mescher~\cite{FM}. In the light of the formula $\cd(\pi\times\pi)=2\cd(\pi)$~\cite{Dr1}
their result states that for hyperbolic groups with $\cd(\pi)=n$ either $\TC(\pi)=2n$ or $\TC(\pi)=2n-1$.

In the paper we use the following properties of torsion free hyperbolic groups: The centralizer of every nontrivial element is cyclic and a finitely generated torsion free hyperbolic group $\pi$  has finite classifying complex $B\pi$.
Our main tool  is the characterization of  $\TC(\pi)$ given by Farber-Grant-Lupton-Oprea~\cite{FGLO}:
\begin{thm}\label{FGLO}
The topological complexity $\TC(\pi)$ of a  group $\pi$ with finite $B\pi$ equals the minimal number $k$ such that the canonical map $f:E(\pi\times\pi)\to E_{\mathcal D}(\pi\times\pi)$ admits
a $(\pi\times\pi)$-equivariant deformation to the $k$-skeleton $E_{\mathcal D}(\pi\times\pi)^{(k)}$.
\end{thm}
Here $E_{\mathcal F}G$ is the notations for a classifying space for $G$-action with isotropy groups in the class
of subgroups $\mathcal F$~\cite{Lu}. The class $\mathcal D$ is defined by conjugations of the diagonal subgroup $\Delta(\pi)\subset\pi\times\pi$ and their intersections.

It turns out to be that the essentiality of $f$ in the top dimension can be detected cohomologically.  In the case when $B\pi$ is an oriented manifold this can be done
using the induced map of the orbit spaces $\bar f:B(\pi\times\pi)\to B_{\mathcal D}(\pi\times\pi)$. Generally, our argument uses compactifications of the classifying spaces and
the main result of~\cite{Dr1}.

\subsection{Notations and conventions}
In this paper a group $\pi$ that admits a finite classifying space $B\pi$ is called geometrically finite. Hyperbolic groups in the paper a Gromov's hyperbolic~\cite{Gr}.

The notations for cohomology of a group is $H^*(\pi,M)$ and for
cohomology of a space $H^*(X;M)$.

We consider an actions of a discrete group $G$ on CW-complex $X$ that preserve the CW-complex structure in such a way that if an open cell $e$ is taken by an element $g\in G$ to itself
then $g$ restricted to $e$ is the identity map. Such a complex is called a $G$-CW-complex.

\section{Classifying space $E_{\mathcal D}G$}

For a discrete group $G$ and a family of its subgroups $\mathcal F$ which is closed under conjugation and finite intersection a classifying complex $E_{\mathcal F}G$
is a $G$-CW-complex whose isotropy subgroups are from $\mathcal F$ and for each $G$-CW complex with isotropy groups from $\mathcal F$ there is a uniqe  up to a 
$G$-homotopy a $G$-equivariant map $f:X\to E_{\mathcal F}G$. This implies in particular,that any two such complexes a $G$-homotopy equivalent. 
Note that the universal cover $EG$ of the classifying space $BG$ is a $G$-CW-complex. The corresponding map $$f:EG\to E_{\mathcal F}G$$ is 
called a {\em canonical map}.

It was proven in~\cite{Lu}
that a $G$-CW-complex $X$ is $E_\mathcal FG$ if and only if for each group $F\in\mathcal F$ the fixed point set $X^F$ is contractible.

\subsection{The family $\mathcal D$ for groups with cyclic centralizers}
Let $\pi$ be a torsion free discrete group, we set $G = \pi \times \pi$. We
denote by $\mathcal D$ the smallest family of subgroups $H\subset G$ which contains the diagonal
$\Delta(\pi)\subset\pi\times\pi$, the trivial subgroup and which is closed under  conjugations and finite
intersections. A complete description of the family $\mathcal D$ is given in 4.1 of~\cite{FGLO}. 
Let $\mathcal D'$ denote the set of all groups in $\mathcal D$ except the trivial group.

Also
it was proven in~\cite{FGLO} that for a torsion free group $\pi$ with cyclic centralizers for any two centralizers $Z(a)$, $Z(b)$ of elements $a,b\in \pi$ either
$Z(a)=Z(b)$ or $Z(a)\cap Z(b)=\{e\}$. Their description of $\mathcal D$ in the case of cyclic centralizers can be stated as follows:
\begin{prop}\label{D}
For a torsion free group $\pi$ with cyclic centralizers  a subgroup $H\subset \pi\times\pi$ belongs to $\mathcal D'$ if and only if it is of the form
$$
(\ast)  \ \ \ \ H_{\gamma,b}=(\gamma,e)\Delta(Z(b))(\gamma^{-1},e)
$$
where $\gamma, b\in \pi$ and $\Delta:\pi\to\pi\times\pi$ is the diagonal map.
\end{prop}
Thus, the family $\mathcal D'$ for a torsion free group with cyclic centralizers consists of the conjugates of the diagonal subgroup $\Delta(\pi)\subset\pi\times\pi$ and the conjugates of the centralizers 
of all nontrivial elements in $\Delta(\pi)$.

\begin{prop}\label{equal}
Given $\gamma, b,\alpha, c\in\pi$, either $H_{\gamma,b}=H_{\alpha, c}$ or $H_{\gamma,b}\cap H_{\alpha, c}=(e,e)$.
\end{prop}
\begin{proof}
If $(e,e)\ne (x,y)\in H_{\gamma,b}\cap H_{\alpha, c}$, then $c\in Z(b)$ and hence $H_{\alpha,c}=H_{\alpha, b}$. Also, $\gamma b^k\gamma^{-1}=\alpha b^k\alpha^{-1}$
for some $k$. Hence $\alpha^{-1}\gamma\in Z(b)$ and hence, $\gamma=\alpha b^n$ for some $n$. Then $H_{\gamma,b}=H_{\alpha, b}$.
\end{proof}

We call two elements $a,b\in\pi$ {\em weakly conjugate} if there are integers $k$ and $\ell$ such that $a^k$ is conjugate to $b^\ell$. Clearly this is an equivalence relation.
Note that for a torsion free group the unit form a separate equivalence class.
Let $C$ be the set of the weak conjugacy classes of nontrivial elements in $\pi$. Let $A$ be a section of $C$, i.e. a collection of elements $a_F\in F\in C$ indexed by $C$.
We may assume that $a$ is a generator of $Z(a)\cong\Z$. If not, then $Z(a)$ is generated by some $t$ and $a=t^n$ lies in the same weak conjugacy class.
So we can choose $t$ instead of $a$.

\begin{prop}
The set of groups $\mathcal D'$ can be presented as the disjoint union 
$$\mathcal D'=\coprod_{a\in A}\{g\Delta(Z(a))g^{-1}\mid g\in G\}$$ of  families of groups indexed by $a\in A$ where each family consists of all distinct conjugates of
$\Delta(Z(a))$.
\end{prop}
\begin{proof}
Note that every conjugate $$(x,y)\Delta(Z(a))(x^{-1}y^{-1})=(xy^{-1},e)\Delta(yZ(a)y^{-1})((xy^{-1})^{-1},e),$$ 
$(x,y)\in\pi\times\pi=G$,
is one of the groups $H_{\gamma, b}$ from Proposition~\ref{D}.
Suppose that $b$ belongs to the weak conjugacy class of $a$, $b^k=xa^\ell x^{-1}$. Then $$(\gamma,e)(b^k,b^k)(\gamma^{-1},e)=(\gamma x,x)(a^\ell,a^\ell)((\gamma x)^{-1},x^{-1})\in 
(\gamma x,x)\Delta(Z(a))(\gamma x,x)^{-1}.$$ 
By Proposition~\ref{equal} $H_{\gamma,b^k}$ coincides with $(\gamma x,x)\Delta(Z(a))(\gamma x,x)^{-1}$. Since $Z(b)=Z(b^k)$, we obtain $H_{\gamma,b}=H_{\gamma,b^k}$
and hence,
$$H_{\gamma,b}\in\{g\Delta(Z(a))g^{-1}\mid g\in G\}.$$

Now we show that the families in the union are disjoint. Suppose that $$g\Delta(Z(a))g^{-1}= h\Delta(Z(b))h^{-1}.$$  Then  $\Delta(Z(a))= g^{-1}h\Delta(Z(b))h^{-1}g$.
Let $g^{-1}h=(x,y)$. Then  the condition $$g^{-1}h\Delta(Z(b))h^{-1}g\subset \Delta(\pi)$$ implies $xbx^{-1}=yby^{-1}$. Therefore, $x^{-1}y\in Z(b)$ and, hence,
$y=xb^k$ for some $k$. Then
$$(x,y)\Delta(Z(b))(x^{-1},y^{-1})=(x,x)(e,b^k)\Delta(Z(b))(e,b^{-k})(x^{-1},x^{-1})=$$ $$\Delta(xZ(b)x^{-1})=
\Delta(Z(xbx^{-1}))=\Delta(Z(a)).$$  
The last equality implies that $xbx^{-1}=a^n$ for some $n$. Thus, $b$ is weakly equivalent to $a$.
\end{proof}
We note that the unit $e\in\pi$ form a separate conjugacy class with $Z(e)=\pi$. Thus, the summand corresponding to $e\in A$ consists of conjugates of the diagonal subgroup.

\subsection{Construction of the classifying space $E_{\mathcal D}G$}
Assume that a group $G$ acts on a space $X$ and $A$ is a $G$-invariant closed subset. Let $\phi:A\to Y$ be a $G$-equivariant map.
Then, clearly, the space $Y\cup_\phi X$ obtained from attaching $X$ to $Y$ along $A$ admits a natural $G$-action. 
\begin{defin}
Let $X$ be a CW-complex with a $G$-action preserving the CW-complex structure. Let $\phi:\partial D^n\to X^{(n-1)}$ be an attaching map. We say that the complex $$Z=X\cup_{\bigcup g\phi}(\coprod_{g\in G}D^n)$$ {\em is obtained from $X\cup_\phi D^n$  by the $G$-translation}. Thus, the action of $G$ on $X$ extends to $Z$.
\end{defin}

Let $G=\pi\times\pi$ and let $E\pi$ be the universal covering of a classifying complex $B\pi$ of the group $\pi$. 
As the above we assume that $\pi$ is torsion free and it has centralizers of all nontrivial elements cyclic.
Note that the translates of the diagonal $\Delta=\Delta(E\pi)\subset E\pi\times E\pi$ by elements of $g\in G$ are either disjoint or coincide.
We consider the quotient space
$$
X=(E\pi\times E\pi)/\{g\Delta\}_{g\in G}
$$
obtained from the product $E\pi\times E\pi$ by collapsing to points all translates of the diagonal.
We may assume that $EG$ has a $G$-equivariant CW-complex structure such that $\Delta$ is a subcomplex. 
Then $X$ has the quotient CW-complex structure. Note that the distinct translates $g\Delta$ of $\Delta$
can be indexed by elements of $\pi$
in the following way $\{(\gamma,e)\Delta\}_{\gamma\in\pi}$. We denote the image of of $(\gamma,e)\Delta$ in $X$ by $v_\gamma$  and the image of $\Delta$ by $v=v_e$. Thus, $v_\gamma=(\gamma,e)v$
are  vertices in $X$.
Note that $X$ is contractible and the action of $G$ on $E\pi\times E\pi$ induces a $G$-action on $X$. 

 Let $a$ be a generator
of the centralizer $Z(a)\cong \Z$, $a\in \pi$
Then the set $\{(a^n,e)v\mid n\in\Z\}$ is the fixed point set for the group $\Delta(Z(a))$
and the set $$\{g(a^n,e)(v)\mid n\in\Z\}$$ is the fixed point set of the group  $g\Delta(Z(a))g^{-1}$ for $g\in G$.

We construct $E_{\mathcal D}G$ by attaching to $X$ cells of dimensions 1, 2, and 3.

 We consider the natural action of $G$ on the set of unordered pairs
$\{\{u,v\}\mid u,v\in Gv\}$. 
\begin{prop}
For any $a\in A$ the stabilizer of the pair $\sigma=\{v,v_a\}$ is the centralizer $\Delta(Z(a))$.
\end{prop}
\begin{proof}
Suppose that $(x,y)\{v,v_a\}=\{v,v_a\}$ $x,y\in\pi$. There are two cases: 
$(x,y)v=v$ and $(x,y)v_a=v_a$  or  $(x,y)v=v_a$ and $(x,y)v_a=v$.

In the first case,  from the equality $(x,y)v=(xy^{-1},v)v=v$ we obtain that $xy^{-1}=e$ and, hence, $x=y$. From the condition 
$(x,y)v_a=v_a$ we obtain $(x,y)v_a=(x,x)(a,e)v=(xa,x)v=(xax^{-1},e)v=(a,e)v$ and, hence, $xa^{-1}x=a$. Thus $x\in Z(a)$ and $(x,y)\in\Delta(Z(a))$.

In the second case we obtain the pair $(x^2,y^2)$ satisfies the conditions of the first case. Therefore, $x^2=y^2$ and $x^2,y^2\in Z(a)$. The latter implies that $x\in Z(a)$ and $y\in Z(a)$ (see~\cite{FGLO}).
\end{proof}
\begin{cor}
For every $n\in \Z$ the stabilizer of the pair $\sigma_n=\{v_{a^n},v_{a^{n+1}}\}$ is $\Delta(Z(a))$.
\end{cor}
\begin{proof}
We note that $\sigma_n=(a^n,e)\sigma$. Hence the stabilizer of $\sigma_n$ is the conjugate of the stabilizer of $\sigma$, $$(a^n,e)\Delta(Z(a))(a^{-n},e)=\Delta(Z(a)).$$
\end{proof}

For a fixed $a\in A$ we construct a $G$-equivariant CW complex $X_a$ containing $X$ as an invariant subset. 
Let the pair $\sigma=\{v,v_a\}$ also denote the interval spanned by $\{v,v_a\}$. We consider such an interval for each pair of points in the orbit $G\{v,v_a\}$.
Thus, we obtain a $G$-equivariant family of intervals attached to $X$
such that no two interval have the same set of vertices. We denote a new complex by $X_a^1$. The action of $G$ on $X$ extends to $X_a^1$. 
Note that in $X^1_a$ the intervals $(a^n,e)\sigma$, $n\in\Z$, form a real line $R_a\cong\R$ attached to $X$ along the integers 
$\Z\subset\R$ to the vertices  $v_{a^n}$. Then $R_a$ is the fixed point set 
for $\Delta(Z(a))$ and it is contractible. Moreover, each group $g\Delta(Z(a))g^{-1}$, $g\in G$ has the contractible fixed point set $gR_a$.

Next we will make $X_a^1$ simply connected. For that we choose a subcomplex $I\subset X$ homeomorphic to the interval $[0,1]$ connecting the vertices $v$ and $v_a$ and attach a 2-disk $D$ to the circle $I\cup\sigma$. Then we define $X_a^2$ to be a $G$-complex obtained from $X_a^1\cup D$ by the $G$-translation.

The disks $D$ and  $(a,a)D$ together with a contractible space $X$ form in $X_a^2$ a unique up to homotopy  2-spheroid $S_n$. We fill it with a 3-ball $B$, 
using an attaching map for $$\phi:\partial B\to D\cup(a,a)D\cup X^{(2)}.$$ We define a $G$-complex $X_a$ as the complex obtained from $X^2_a\cup_\phi B$ by the $G$-translation.

\begin{prop}
The CW complex $X_a$ has the following properties: 
\begin{itemize}
\item[1.] $X_a$ is $G$-equivariant;

\item[2.] $X_a$ is contractible;

\item[3.] The groups $g\Delta(Z(a))g^{-1}$ have contractible fixed point sets;

\item[4.] $\dim(X_a\setminus X)=3$.
\end{itemize}
\end{prop}
\begin{proof}
The conditions 1, 3, and 4 are obvious. We prove 2.

Let $Y_a$ be a CW-complex obtained by the above procedure from $X$ by attaching the cells $\sigma$, $D$, and $B$ with the translations by the group $\Delta(Z(a))\cong \Z$.
Since $\sigma$ is fixed by $\Delta(Z(a))$, the complex  $Y_a$ has one new 1-cell $\sigma$. In dimensions 2 and 3, it has cells indexed by $\Z$ such that the 
quotient CW complex $Y_a/(X\cup D)$ isomorphic to the reduced double suspension over the reals $\R$ with $\Z$ as the 0-skeleton. Therefore, $Y_a$ is contractible.
Note that 
$$X_a=\bigcup_{[g]\in G/\Delta(Z(a))} gY_a$$ 
with $X$ being the common part of all summands $gY_a$. Hence $X_a$ is contractible.
\end{proof}

We perform this construction for all $a\in A$ and define $$X_3=\bigcup_{a\in A}X_a$$ which is obtained from the disjoint union $X_3=\coprod_{a\in A}X_a$
by the identification of all copies of $X\subset X_a$.

\begin{prop}\label{no-intersection}
For $a,b\in A$, $a\ne b$,
The orbits   of the pairs $G\{v,v_a \}$ never lands in $R_b$.
\end{prop}
\begin{proof}
	Suppose $g(v)=v_{b^k}$ and $g(v_a)=v_{b^m}$.  Let $g=(x,y)\in\pi\times\pi$. The first equality implies that
$(x,y)v=(xy^{-1},e)v=(b^k,e)v$ and, hence, $x=b^ky$. The second equality implies $(x,y)(a,e)v=(xa,y)v=(xay^{-1},e)v=(b^m,e)v$ and, therefore, $xay^{-1}=b^m$.
Thus, $b^kyay^{-1}=b^m$ which implies that $a$ and $b$ are weakly conjugate. This contradicts to the choice of $a$ and $b$.
\end{proof}

\begin{prop}
The complex $X_3$ satisfies all the conditions of $E_{\mathcal D}G$.
\end{prop}
\begin{proof}
The CW complex $X_3$ is a $G$-complex as the union of $G$-complexes with a $G$-invariant intersection.
Note that $G$ acts freely on $$X_3\setminus(\bigcup_{a\in A}GR_a).$$  Already we know that the vertices of the fixed point complex of $\Delta(Z(b))$ are those of $R_b$.
In view of Proposition~\ref{no-intersection}, no new edges are spanned by vertices in $R_b$ and hence the fixed point set of $\Delta(Z(b))$ is just $R_b$ which is contractible.
Therefore, the fixed point sets are contractible for all groups $H\in\mathcal D$.

The complex $X_3$ is contractible as a union of contractible complexes $X_a$ with a contractible intersection.
\end{proof}

\subsection{CW structure of $B_{\mathcal D}G$}
Note that $B_{\mathcal D}G=E_{\mathcal D}/G$ contains $(B\pi\times B\pi)/\Delta(B\pi)$ as a subcomplex with a 3-dimensional complement.
Moreover, as a CW complex $$B_{\mathcal D}G=\bigcup_{a\in A}X_a/G=(B\pi\times B\pi)/\Delta(B\pi)\cup \bigcup_{a\in A}(e^1\cup e^2\cup e^3)$$ 
where $e^1$ is the image of $Int(\sigma)$, $e^2$ is the image of $Int(D)$, and
$ e^3$ is the image of the interior of the 3-ball $B$ from $X_a$. Since $D$ can be deformed to $I$, the complex $B_{\mathcal D}G$ is homotopy equivalent to
$$
(B\pi\times B\pi)/\Delta(B\pi)\cup \bigcup_{a\in A} e^3$$ 
where 3-cells are attached to the image of $I$. Hence the attaching maps are null-homotopic. Thus, the space $B_{\mathcal D}G$ is homotopy equivalent to the bouquet of $
(B\pi\times B\pi)/\Delta(B\pi)$ with
a wedge of 3-spheres,
$$
B_{\mathcal D}G\sim (B\pi\times B\pi)/\Delta(B\pi)\vee \bigvee_{a\in A}S^3.$$

\section{Conditions when $\TC(\pi)$ is maximal}

We use the following 
\begin{thm}[\cite{Dr1}]\label{square}
For a geometrically finite group $\pi$ with $\cd(\pi)=n$,
$$
H^{2n}(\pi\times\pi,\Z(\pi\times\pi))\ne 0.
$$
\end{thm}
We note that for a group $G$ with finite $BG$ there is the equality~\cite{Bro} $$H^k(G,\Z G)=H_c^k(EG;\Z)$$ where $EG$ is the universal cover of $BG$
and $H_c^*$ stands for cohomology with compact supports.

\

Here is our main result.
\begin{thm}
Suppose that a geometrically finite group $\pi$ has cyclic centralizers and  $\cd\pi=n\ge 2$. Then $\TC(\pi)=2n$.
\end{thm}
\begin{proof}
Let $G=\pi\times\pi$ and let $E=E\pi\times E\pi$ where  $E\pi$ is the universal cover of a finite simplicial complex $B\pi=E\pi/\pi$.
Let $B=BG=E/G$.
Suppose that the canonical map $f:E\to X_3$ can be deformed by a $G$-homotopy $\Phi:E\times I\to X_3$ to a map with the image in $X_3^{(2n-1)}$.
 Let $$Z=\bigcup_{g\in G}g\Delta(E\pi)\cup E^{(2)}.$$
We use the notation $\alpha Y$ for the one-point compactification of a locally compact space $Y$.
Since $\dim\alpha Z=\max\{n,2\}< 2n-1$, the exact sequence of pair implies that $H^{2n}(\alpha E)=H^{2n}(\alpha E,\alpha Z)$. Therefore, the collapsing map
$\bar q:\alpha E\to\alpha E/\alpha Z$ induces an isomorphism of cohomology groups
$$\bar q^*:H^{2n}(\alpha E/\alpha Z;\Z)\to H^{2n}(\alpha E;\Z).$$

Consider the commutative diagram
\[
\xymatrix{& X_3\ar[d]^{\psi}\\
 E\ar[ur]^f\ar[r]^q \ar[d]^{\subset}
& E/Z \ar[d]^{\xi}\\
 \alpha E \ar[r]^{\bar q}
& \alpha E/\alpha Z}
\]
where $\psi$ is defined by sending $X_3\setminus X$ to the point in $E/Z$ generated by $Z$. The map $\xi$ is a continuous bijection which is a homeomorphism in the complement of the singular one-point sets defined by $Z$ and $\alpha Z$.

Let $\alpha=\alpha E\setminus E$ denote the point at infinity.
We claim that sending 
$\alpha\times I$ to the singular point $z_0\in\alpha E/\alpha Z$ continuously extends
the homotopy $$\xi\circ\psi\circ\Phi:E\to\alpha E/\alpha Z$$ to a continuous homotopy $\hat\Phi:\alpha E\times I\to \alpha E/\alpha Z$.

Assume the contrary: There is an open neighborhood $U$ of $z_0$ and a sequence $x_n\in E$ converging to $\alpha$ such that
$$\xi\psi\Phi(x_n\times I)\cap(\alpha E/\alpha Z)\setminus U\ne\emptyset$$
for all $n$. Then $\psi\Phi(x_n\times I)\cap C\ne\emptyset$ for all $n$ where $$C=\xi^{-1}((\alpha E/\alpha Z)\setminus U)$$ is a compact subset of $E\setminus Z\subset E/Z$.
Fix $t_n\in I$ such that $\psi\Phi(x_n,t_n)\in C$.
 Let $p:E\to B$ be the universal covering map and let  $$\bar p:E/Z\to B/p(Z)$$ be the induced map. We consider a metric $d_E$ on $E$ lifted from a metric on $B$.
Let $V$ be an open neighborhood of the singular point $\bar b_0\in B/p(Z)$
such that $$C\cap \bar p^{-1}(\bar V)=\emptyset$$ where $\bar V$ is the closure of $V$. 
We may assume that the quotient map $\nu: B\to B/p(Z)$ is 1-Lipschitz. Consider the metric $d'$ on $W=E/Z\setminus\bar p^{-1}(V)$  lifted with respect to $\bar p$
from the matric on $B/p(Z)\setminus V$. 
Note that the action of $G$ on $W$ is proper discontinuous by isometries.
Let $\epsilon=dist_{d'}(C,p^{-1}\partial V)>0$.
The homotopy $\psi\Phi$ defines a homotopy $\hat\Phi:B\times I\to B/p(Z)$ such that the diagram
$$
\begin{CD}
E\times I @>\phi\Phi>> E/Z\\
@Vp\times 1VV @VV\bar pV\\
B\times I @>\hat\Phi>> B/p(Z)\\
\end{CD}
$$
commutes.
Then the restriction $$\psi\Phi|_{W'}:W'=(E\times I)\setminus p^{-1}\nu^{-1}(V)\to W$$ is  Lipschitz.

Compactness of $B$ and $I$ implies that there is  a subsequence $(x_{n_k},t_{n_k})$ such that $\{p(x_{n_k})\}$ converges to some $\bar x\in B$ and $\{t_{n_k}\}$
converges to some $t^*\in I$. Note that $\bar\Phi(\bar x,t^*)\in\bar p(C)$.
 Let $p(x)=\bar x$. Then $G(x)\times t^*\subset W'$.
By the choice of the metric on $E$ there is  a sequence of elements $g_k\in G$, $\|g_k\|\to\infty$,
such that $$d_E(x_{n_k},g_kx)\to 0.$$

Since $\bar p(C)\cap \bar V=\emptyset$, for sufficiently large $k$ we have $Gx\times t_{n_k}\subset W'$.
Then the Lipschitz condition of the restriction of $\psi\Phi$ to $W'$ implies 
$$d'(\psi\Phi(x_{n_k},t_{n_k}),\psi\Phi(g_k(x),t_{n_k}))<\epsilon/2$$  for sufficiently large $k$ .

Then $$d'(\psi\Phi(g_kx,t^*),C)\le d'(\psi\Phi(g_kx,t^*),\psi\Phi(g_kx,t_k)) + d'(\psi\Phi(g_kx,t_k),\psi\Phi(x_{n_k},t_k)).$$
For sufficiently large $k$ each of the above summands does not exceed $\epsilon/2$.
Thus, the sequence $$g_k(\psi\Phi(x,t^*))=\psi\Phi(g_kx,t^*)$$ stays in a bounded  distance from a compact set.
This contradicts with the properness of the action of $G$ on $W$.

Thus, there is a homotopy $$\hat\Phi:\alpha E\times I\to \alpha E/\alpha Z$$ of $\bar q$ to a map with the image in the $(2n-1)$-dimensional compact space $\alpha E^{(2n-1)}/\alpha Z$.
Therefore, $\bar q$ induces zero homomorphism of $2n$-cohomology. Therefore, $$H^{2n}(G,\Z G)=H^{2n}_c(E;\Z)=H^{2n}(\alpha E;\Z)=0.$$
This contradicts to Theorem~\ref{square}.
\end{proof}

\section{Maximal value of TC via cohomology of the orbit spaces}

In this section we present a shorter proof of the main theorem for Poincare Duality groups $\pi$.
\begin{prop}\label{diag}
Suppose that for a geometrically finite group $\pi$ with $\cd\pi=n\ge 2$ and cyclic centralizers, $H^{2n}(B\pi\times B\pi;q^*\mathcal A)\ne 0$ for some local system of coefficients $\mathcal A$ on the quotient space
$(B\pi\times B\pi)/\Delta(B\pi)$ where $q:B\pi\times B\pi\to (B\pi\times B\pi)/\Delta(B\pi)$ is the quotient map. Then $\TC(\pi)=2n$.
\end{prop}
\begin{proof}
The exact sequence of pairs form a commutative diagram
$$
\begin{CD}
 H^{2n}(B\pi\times B\pi;q^*\mathcal A) @<\cong<< H^{2n}(B\pi\times B\pi,\Delta(B\pi);q^*\mathcal A) @<<< 0\\
@ Aq^*AA  @A\cong AA @.\\
 H^{2n}(B\pi\times B\pi/\Delta(\pi);\mathcal A) @<\cong<< H^{2n}(B\pi\times B\pi/\Delta(B\pi),pt;\mathcal A) @<<< 0\\
\end{CD}
$$
where the right vertical homomorphism is an isomorphism in view of excision~\cite{Bre}. Thus, $q^*$ is an isomorphism.
Therefore $q$ cannot be deformed to the $2n-1$-skeleton in $(B\pi\times B\pi)/\Delta(B\pi)$ as well as in $B_{\mathcal D}G$. Hence $\tilde q:E\pi\times E\pi\to E_{\mathcal D}G$ cannot be $(\pi\times\pi)$-equivariantly deformed to the $(2n-1)$-skeleton.
By Theorem~\ref{FGLO}, $\TC(\pi)\ge 2n$.
\end{proof}
\begin{cor}\label{new for manifolds}
Let a group $\pi$ be with cyclic centralizers and with the classifying space $B\pi$ be a manifold. Then $\TC(\pi)=2\cd(\pi)$.
\end{cor}
\begin{proof}
In the case of orientable manifold $B\pi$ we apply Proposition~\ref{diag} by taking the  integers as coefficients.

If $B\pi$ is non-orientable, we consider the tensor product  $\mathcal O\hat\otimes\mathcal O$ of the orientation sheaves coming from different factors.
Since the action of the diagonal subgroup is trivial on the stalk $\Z\otimes\Z$, it follows that $\mathcal O\hat\otimes\mathcal O=q^*\mathcal A$
for some local system $\mathcal a$ on $(B\pi\times B\pi)/\Delta(B\pi)$ (see~\cite{Dr2}).
\end{proof}

\end{document}